\documentclass[11pt]{article}
\usepackage{amsmath,amsthm,amssymb}
\usepackage[left=3cm,right=3cm, top=2.5cm,bottom=2.5cm,bindingoffset=0cm]{geometry}

\usepackage{amscd}

\usepackage{hyperref}

\usepackage{mathtools}

\mathtoolsset{showonlyrefs}

\usepackage{MnSymbol} 

\usepackage{xcolor}

\usepackage{wrapfig}
\usepackage{graphicx}

\makeatletter
\def\th@plain{%
  \thm@notefont{}
  \itshape 
}
\def\th@definition{%
  \thm@notefont{}
  \normalfont 
}
\makeatother

\newtheorem{lemma}{Lemma}[section]
\newtheorem{proposition}[lemma]{Proposition}

\newtheorem{remark-definition}[lemma]{Remark-Definition}
\newtheorem{theorem}[lemma]{Theorem}
\newtheorem{corollary}[lemma]{Corollary}

\newtheorem{proposition-conjecture}[lemma]{Proposition-conjecture}

\theoremstyle{definition}
\newtheorem{example}[lemma]{Example}

\newtheorem{definition}[lemma]{Definition}
\newtheorem{remark}[lemma]{Remark}

\newcommand{\Ker}{\mathrm{Ker}\,}

\newcommand{\St}{\mathrm{St}}

\newcommand{\rk}{\mathrm{rk}\,}

\newcommand{\gl}{\mathrm{gl}\,}

\newcommand{\const}{\mathrm{const}}

\newcommand{\goth}{\mathfrak}

\newcommand{\codim}{\mathrm{codim}\,}

\newcommand{\Sing}{\mathsf{Sing}}

\newcommand{\svert}{{\mathsf v}_{\mathrm{tot}}}

\newcommand{\shor}{{\mathsf h}_{\mathrm{tot}}}

\newcommand{\nv}{n_{{v}}}

\newcommand{\nh}{n_{{h}}}

\newcommand{\dimO}{\dim \mathcal O_{\mathrm{reg}}}

\newcommand{\codimO}{\codim \mathcal O_{\mathrm{reg}}}

\newcommand{\dimSt}{\dim \mathrm{St}_{\mathrm{reg}}}

\newcommand\trdeg{\mathrm{tr.deg.}\,}

\ifdefined\C
 \renewcommand{\C}{\mathbb{C}}
\else
 \newcommand{\C}{\mathbb{C}}
\fi

\newcommand\CP{\mathbb{C}\mathbb{P}}

\ifdefined\N
 \renewcommand{\N}{\mathbb{N}}
\else
 \newcommand{\N}{\mathbb{N}}
\fi

\newcommand{\charp}{\chi}

\newcommand{\Hom}{\mathrm{Hom}}
\renewcommand{\Im}{\mathrm{Im}\,}

\newcommand{\ivan}[1]{\textit{\color{blue}#1}}

\newcommand\g{\goth{g}}

\usepackage[numbers, sort]{natbib}

\usepackage{mathtools}
\mathtoolsset{showonlyrefs}

\title{Jordan--Kronecker invariants of Lie algebra representations and degrees of invariant polynomials}
\author{Alexey Bolsinov\thanks{School of Mathematics,
 Loughborough University and Faculty of Mechanics and Mathematics, Moscow State University. E-mail: {\tt A.Bolsinov@lboro.ac.uk}},  Anton Izosimov\thanks{Department of Mathematics, University of Arizona. E-mail:  {\tt izosimov@math.arizona.edu}}, and Ivan Kozlov\thanks{Faculty of Mechanics and Mathematics, Moscow State University. E-mail: {\tt ikozlov90@gmail.com} }
}

\begin{document}

\date{}

\maketitle

\abstract{For an arbitrary representation $\rho$ of a complex finite-dimensional Lie algebra, we construct a collection of numbers that we call the \textit{Jordan-Kronecker invariants} of $\rho$. Among other interesting properties, these numbers provide lower bounds for degrees of polynomial invariants of $\rho$. Furthermore, we prove that these lower bounds are exact if and only if the invariants are independent outside of a set of large codimension.  Finally, we show that under certain additional assumptions our bounds are exact if and only if the algebra of invariants is freely generated. }

\tableofcontents

\section{Introduction}
\label{intro}

The main idea of this paper has its roots in the theory of bi-Hamiltonian systems where it was discovered that
the algebraic structure of a pair of compatible Poisson brackets $\{\,,\}_0$ and $\{\,,\}_1$ essentially affects the differential geometry of the pencil $\{\,,\}_0 + \lambda \{\,,\}_1$ and even the dynamical properties of bi-Hamiltonian systems related to it (see, e.g., \cite{turiel, Zakh, Panas, BolsOsh, BolsIzos}).  This observation has recently been used in \cite{BolsZhang} to introduce Jordan--Kronecker invariants for a finite-dimensional Lie algebra $\goth g$, which are directly related to a natural pencil of compatible Poisson brackets on $\goth g^*$. From the algebraic viewpoint, this construction is based on a simple fact that every element $x\in \goth g^*$ defines a natural skew-symmetric bilinear form $\mathcal A_x(\xi,\eta)=\langle x,[\xi,\eta]\rangle$ on $\goth g$. The Jordan--Kronecker invariant of $\goth g$ is, by definition, the algebraic type of the pencil of forms $\mathcal A_{a+ \lambda b}$ for a generic pair $(a,b)\in \goth g^*\times \goth g^*$. 
All possible algebraic types are described by the Jordan--Kronecker theorem on a canonical form of a pencil of skew-symmetric matrices (see, e.g., \cite{Thompson, Gantmaher88}) and, in each dimension, there are only finitely many of them.

In the present paper, this construction is generalized to arbitrary finite-dimensional representations of finite-dimensional Lie algebras. In particular, we show how the classical theorem on the canonical form of a pair of linear maps can be applied in the study of Lie algebra representations and their invariants. In what follows, we assume that all objects are defined over the field $\C$ of complex numbers, although everything can be generalized to the case of an arbitrary field of characteristic zero (in which case some of the objects we consider are defined over the algebraic closure of the initial field).

Our construction can be outlined as follows. Let  $\rho :
\mathfrak{g} \to \gl (V)$ be a linear representation of a finite-dimensional Lie algebra $\mathfrak{g}$ on a finite-dimensional vector space $V$. To this representation and an arbitrary element $x \in  V$, one can naturally associate an operator $R_x :
\mathfrak{g} \to V$ defined by $R_x (\xi) = \rho(\xi)
x$. Consider a pair of such operators $R_a, R_b$ and the pencil $R_a+\lambda R_b=R_{a+\lambda b}$ generated by them. It is well known that such a pencil can be completely characterized by a collection of quite simple numerical invariants
 (see Section~\ref{S:JK_Operator} for details). In the present paper we show that many important and interesting properties of the representation $\rho : \mathfrak{g} \to \gl (V)$ are related to and can be derived from the invariants of such a  pencil $R_{a+\lambda b}$ generated by a generic pair $(a,b)\in V\times V$. In particular, one of our main results gives lower bounds for degrees of invariant polynomials of the representation $\rho$ in terms of the numerical invariants of the associated pair of operators $R_a, R_b$ (Theorem~\ref{T:SumDeg_Polyn}). Furthermore, we show that these lower bounds are exact if and only if the invariant polynomials of $\rho$ are independent outside of a subset of codimension $\geq 2$ (Theorem \ref{thm2}). Finally, we prove that under certain additional assumptions our bound is exact if and only if the algebra of invariant polynomials of $\rho$ is freely generated (Theorem~\ref{thm3}). This generalizes several previously known results (see, in particular, \cite{Ooms2, JS, OVDB}) which relate polynomiality of the algebra of invariants with the \textit{sum} of degrees of the invariants (while our statements concern \textit{individual} invariants). In the last Section \ref{S:SumsStandRepr} we show that our results have quite non-trivial implications already in the case of standard representations of simple Lie groups. 

\par\medskip
 
\textbf{Acknowledgements} The work of A.\,V.~Bolsinov and I.\,K.~Kozlov was supported by the Russian Science Foundation (project No.17-11-01303).

\section{Canonical form of a pair of linear maps}
\label{S:JK_Operator}

In this section, we recall the normal form theorem for a pair of linear maps.  We first state the theorem in the matrix form, and then discuss the invariant meaning of the ingredients involved.

\begin{theorem}[On the Jordan-Kronecker normal form \cite{Gantmaher88}]
\label{T:JK_operator}
Consider two complex vector spaces
$U$ and $V$. Then for every two linear maps $A, B: U\to V$ there are bases in  $U$ and $V$ in which the matrices of the pencil $\mathcal P=\{A + \lambda B\}$ have the following block-diagonal form:
\begin{equation}
\label{Eq:JK_Operator}
{\footnotesize A + \lambda B =
\begin{pmatrix}
0_{m, n} &     &        &      \\
    & \!\!\! A_1 + \lambda B_1 &        &      \\
    &     & \!\!\!\! \ddots &      \\
    &     &        & \!\!\!  A_k + \lambda B_k  \\
\end{pmatrix},
}
\end{equation}
where $0_{m, n}$ is the zero $m \times n$-matrix, and each pair of the corresponding blocks   $A_i$ and $B_i$ takes one of the following forms:

1. Jordan block with eigenvalue $\lambda_0 \in \mathbb{C}$
 \[
 A_i =\left( \begin{matrix}
   \lambda_0 &1&        & \\
      & \lambda_0 & \ddots &     \\
      &        & \ddots & 1  \\
      &        &        & \lambda_0   \\
    \end{matrix} \right),
\quad  B_i=  \left( \begin{matrix}
    -1 & &        & \\
      & -1 &  &     \\
      &        & \ddots &   \\
      &        &        & -1   \\
    \end{matrix} \right).
\]

2. Jordan block with eigenvalue $\infty$
\[
A_i = \left( \begin{matrix}
   1 & &        & \\
      &1 &  &     \\
      &        & \ddots &   \\
      &        &        & 1   \\
    \end{matrix}  \right),
\quad B_i = \left( \begin{matrix}
    0 & -1&        & \\
      & 0 & \ddots &     \\
      &        & \ddots & -1  \\
      &        &        & 0   \\
    \end{matrix}  \right).
 \]

  3. Horizontal Kronecker block 
\[
A_i= \left(
 \begin{matrix}
    0 & 1      &        &     \\
      & \ddots & \ddots &     \\
      &        &   0    & 1  \\
    \end{matrix}  \right),\quad
B_i = \left(\begin{matrix}
   -1 & 0      &        &     \\
      & \ddots & \ddots &     \\
      &        & -1    &  0  \\
    \end{matrix}  \right).
    \]

   4. Vertical Kronecker block
\[
A_i= \left(
  \begin{matrix}
  0  &        &    \\
  1   & \ddots &    \\
      & \ddots & 0 \\
      &        & 1  \\
  \end{matrix}
 \right), \quad
B_i = \left(
  \begin{matrix}
  -1  &        &    \\
  0   & \ddots &    \\
      & \ddots & -1 \\
      &        & 0  \\
  \end{matrix}
 \right).
\]

 The number and types of blocks in decomposition  \eqref{Eq:JK_Operator} are uniquely defined up to permutation.
 \end{theorem}
\begin{remark}
The negative signs in the matrices $B_i$ are not important and are introduced to simplify some of the formulas to follow.
\end{remark}
\begin{remark}
In what follows, we interchangeably use the following two notions of a pencil:  $\{A + \lambda B\}$ and  $\{ \alpha A + \beta B\}$. These two notions are equivalent if we allow $\lambda$ to be infinite (in which case $A + \lambda B = B$) and do not distinguish between operators that are scalar multiplies of each other. Under this convention, $(\beta : \alpha)$ in $\{ \alpha A + \beta B\}$ are simply homogeneous coordinates of $\lambda \in \bar\C$ in  $\{A + \lambda B\}$ .
%
%
\end{remark}

It is convenient to regard the zero block $0_{m, n}$ in \eqref{Eq:JK_Operator} as a block-diagonal matrix that is composed of $m$ vertical Kronecker blocks of size  $1\times 0$ and $n$ horizontal Kronecker blocks of size $0\times 1$.

\begin{definition}
The \textit{horizontal indices} $\mathsf h_1, \dots, \mathsf h_p$ of the pencil $\mathcal P=\{A+\lambda B\}$ are defined to be the horizontal dimensions (widths) of horizontal Kronecker blocks (i.e. each horizontal index is the number of columns in the corresponding horizontal Kronecker block). Similarly, the \textit{vertical indices} $\mathsf v_1, \dots, \mathsf v_q$  are the vertical dimensions (heights) of vertical blocks.
\end{definition}
In particular, in view of the above interpretation of the $0_{m,n}$ block, the first $m$ horizontal indices and first $n$ vertical indices are equal to $1$. 
We will denote the total number $p$ of horizontal indices by $\nh$, and the total number $q$ of vertical indices by $\nv$.
\begin{remark}\label{minIndices}
There also exist closely related notions of \textit{minimal row and column indices}. Namely, minimal column indices are equal to the numbers $\mathsf h_i -1$, while minimal row indices are equal to $\mathsf v_i - 1$. However, we find the terms \textit{horizontal and vertical indices} more intuitive and more suitable for our purposes.
\end{remark}
\begin{definition}
The total number of columns in horizontal blocks
$\shor = \sum_{i=1}^{\nh} \mathsf h_i$ is said to be  the {\it total Kronecker $h$-index} of the pencil $\mathcal P$. Similarly,  the total number of rows in the vertical Kronecker blocks  $\svert= \sum_{i=1}^{\nv} \mathsf v_j$ is said to be the {\it total Kronecker $v$-index} of $\mathcal P$.
\end{definition}


We now give an invariant interpretation for eigenvalues of Jordan blocks, as well as for vertical and horizontal indices. We begin with Jordan blocks. 
\begin{definition}
The \textit{rank of the pencil} $\mathcal P=\{A+\lambda B\}$ is the number $\rk \mathcal P = \max_{\lambda\in \mathbb C} \rk (A+\lambda B)$. 
\end{definition}
\begin{definition}
The {\it characteristic polynomial}  $\charp(\alpha, \beta)$ of the pencil $\mathcal P$ is defined as the greatest common divisor of all  the $r\times r$ minors of the matrix  $\alpha A + \beta B$,  where 
$r=\rk \mathcal P$.
\end{definition} One can show that the polynomial $\charp(\alpha, \beta)$ does not depend on the choice of bases and therefore is an invariant of the pencil. Furthermore, it is easy to see that $\charp(\alpha, \beta)$ is the product of characteristic polynomials of all the Jordan blocks. These polynomials, in turn, are called {\it elementary divisors} of the pencil and also admit a natural invariant interpretation, see \cite{Gantmaher88} for details.

\begin{proposition}
\label{cor2}
The eigenvalues of Jordan blocks can be characterized as those  $\lambda\in\mathbb \C$ for which the rank of  $A+\lambda B$ drops,  i.e.  $\rk (A+\lambda B) < r=\rk \mathcal P$. The infinite eigenvalue appears in the case when $\rk B <r$. In other words, the eigenvalues of Jordan blocks, written as $\lambda = (\beta : \alpha) \in \CP^1$, are solutions of the characteristic equation $\charp(\alpha, \beta)=0$. Moreover, multiplicity of each eigenvalue coincides with the multiplicity of the corresponding root of the characteristic equation. Jordan blocks are absent if and only if the rank of all non-trivial linear combinations $\alpha A + \beta B$ is the same.

\end{proposition}
\begin{proof}
This is easily verified when $A$ and $B$ are written in the Jordan-Kronecker form~\eqref{Eq:JK_Operator}. The crucial point is that $\rk (A_i+\lambda B_i) = \const$ for any pair $(A_i,B_i)$ of Kronecker blocks, so only Jordan blocks affect the dependence of $\rk (A_i+\lambda B_i)$ on $\lambda$.
\end{proof}
Now we discuss the meaning of vertical and horizontal indices. First of all, the total numbers of horizontal and vertical indices can be defined as the \textit{coranks} of the pencil:
\begin{proposition}
\label{cor1}
Let $\rk P= \max_{\lambda\in \mathbb C} \rk (A+\lambda B)$ be the rank of the pencil $\mathcal P=\{A+\lambda B\}$ of operators $A,B : U \to V$. Then:

\begin{enumerate} \item The number $\nh$ of horizontal indices (or equivalently, the number of horizontal Kronecker blocks)  is equal to $\dim U - \rk P$.

\item The number $\nv$ of vertical indices (or equivalently, the number of vertical Kronecker blocks)  is equal to $\dim V - \rk P$.

\end{enumerate}

In other words, $\nh=\dim\Ker (A+\lambda B)$ and $\nv=\dim\Ker (A+\lambda B)^*$ for generic $\lambda\in\mathbb C$.
\end{proposition}
\begin{proof}
Indeed, for generic $\lambda$ the kernel of $A_i+\lambda B_i$ is one-dimensional if the pair blocks $(A_i,B_i)$ is horizontal Kronecker, and is trivial otherwise. So, $\dim\Ker (A+\lambda B)$ is equal to the number of horizontal blocks. Similarly, $\dim\Ker (A+\lambda B)^*$ is the number of vertical blocks.\end{proof}

We will also need the following relation between the total indices and the degree of the characteristic polynomial:
\begin{proposition} We have
\begin{equation}
\label{sumofsum}
\svert + \shor = \dim V + \dim U - \rk P - \deg \charp.
\end{equation}
\end{proposition}
\begin{proof}
The total numbers columns of matrices $A$ and $B$ is equal to $\dim U$. On the other hand, this number can be computed as the total number of columns in Jordan blocks, which is equal to $\deg \charp$, plus the total number of columns in horizontal Kronecker blocks, which is equal to $\shor$, plus the total number of columns in vertical blocks, which is equal to $\svert - \nv = \svert - \dim V + \rk P$ (see Proposition \ref{cor1}). The result follows.
\end{proof}

Finally, we characterize the horizontal and vertical indices themselves. We begin with the following preliminary statement.

\begin{proposition}\label{prop:kerpencil}
Let $A$ be regular in a pencil $\mathcal P=\{ A+\lambda B\}$, i.e. $\rk A = \rk \mathcal P$. Then for every $u_0\in \Ker A$ there exists a sequence of vectors  $\{u_0, \dots, u_l \in U\}$ such that the expression $u(\lambda)=\sum_{j=0}^{l} u_j
\lambda^j$ is a solution of the equation
\begin{equation}
\label{onceagain}
(A+ \lambda B) u(\lambda)=0.
\end{equation}
Similarly, for any $u_0\in \Ker A^*$ there exists a sequence of vectors  $\{u_0, \dots, u_l \in V^*\}$ such that the expression $u(\lambda)=\sum_{j=0}^{l} u_j
\lambda^j$ is a solution of the equation
\begin{equation}
\label{onceagaindual}
(A+ \lambda B)^* u(\lambda)=0.
\end{equation}

\end{proposition}
\begin{proof}
We prove only the first statement, as the second one is similar. Notice that since $A$ is regular, its kernel is generated by basis vectors in $U$ corresponding to the leftmost columns in horizontal Kronecker blocks.  So, by linearity, it suffices to consider the case when $u_0$ is one of those basis vectors. In this case, as $u_1, u_2, \dots$, one takes the remaining basis vectors in $U$ generating the given horizontal block (which, in particular, means that the number $l$ is equal to one of the horizontal indices minus one).
\end{proof}

\begin{proposition}
\label{L:ChainsRestr}
Let $\mathcal P=\{A+\lambda B\}$ and $A\in\mathcal P$ be regular. Let also
$$
u_i(\lambda) = \sum_{j = 0}^{\deg u_i} u_{ij} \lambda^j,\quad i = 1, \dots, l,
$$
where $u_{ij} \in U$, be polynomial solutions of \eqref{onceagain} such that the vectors $u_i(0) = u_{i0}$ are linearly independent. 
Suppose also that $ \deg u_{1} \leq \dots \leq \deg u_l.$ 
Then,
for each $i=1, \dots, l$ we have \begin{equation}\label{firstDegreeIneq}\deg u_i \geq \mathsf h_i - 1,\end{equation} 
where $\mathsf h_1 \leq \mathsf h_2 \leq \dots$ is the ordered sequence of horizontal indices of $\mathcal P$.

Similarly, if $u_i$'s are polynomial solutions of the dual problem \eqref{onceagaindual}, then \begin{equation}\label{firstDegreeIneqVert}\deg u_i \geq \mathsf v_i - 1,\end{equation} 
where $\mathsf v_1 \leq \mathsf v_2 \leq \dots$ are vertical indices of $\mathcal P$.

%
\end{proposition}
%

Note that the proof of Proposition \ref{prop:kerpencil} provides a way to construct polynomials $u_i(\lambda)$ for which inequalities \eqref{firstDegreeIneq} and \eqref{firstDegreeIneqVert} become equalities. This, along with Proposition \ref{L:ChainsRestr}, immediately implies the following:
\begin{corollary}
\label{L:ChainsRestrCor}
Horizontal  indices $\mathsf h_1, \dots, \mathsf h_p$ are given by $\mathsf h_i = r_i + 1$, where $r_1, \dots, r_p$ are the {minimal} degrees of independent solutions of \eqref{onceagain}. Similarly, vertical  indices $\mathsf v_1, \dots, \mathsf v_q$ are given by $\mathsf v_i = r'_i + 1$, where $r'_1, \dots, r'_q$ are the {minimal} degrees of independent solutions of the dual problem \eqref{onceagaindual}.
\end{corollary}
\begin{remark}
Recall that the numbers $\mathsf h_i -1$, $\mathsf v_i - 1$ are called {minimal indices} (see Remark \ref{L:ChainsRestrCor}). Corollary \ref{L:ChainsRestrCor} explains in which sense these numbers are minimal. 
\end{remark}

\begin{proof}[Proof of Proposition \ref{L:ChainsRestr}]
We only consider the horizontal case, as the vertical one is analogous. 
Begin with the first statement. Let $\tilde u_i(\lambda)$, $i = 1, \dots, \nh$, be polynomial solutions of \eqref{onceagain} constructed in the proof of Proposition \ref{prop:kerpencil}, with $\deg \tilde u_i = \mathsf h_i -1 $. For dimension reasons, these solutions form a basis in the kernel of $A+\lambda B$, when the latter is considered as a matrix over the field of formal Laurent series (we consider Laurent series that are finite in the negative direction). Therefore, we have
\begin{equation}\label{eq:pencilKerBasis}
u_i(\lambda) = \sum_{j = 1}^{\nh} f_{ij}(\lambda) \tilde u_j(\lambda),
\end{equation}
where $f_{ij}$'s are formal Laurent series. Furthermore, from the linear independence of coefficients of $\tilde u_j$'s it follows that $f_{ij}$'s are in fact power series (otherwise the left-hand side would have terms of negative degree in $\lambda$). Now assume that \eqref{firstDegreeIneq} does not hold for some $i = k$, and $k$ is the minimal number with this property. Then
$$
\deg u_1 \leq \dots \leq \deg u_{k} < \mathsf h_k -1 = \deg \tilde u_k  \leq \dots \leq  \deg \tilde u_{p},
$$
which, along with independece of coefficients of $\tilde u_j$'s, implies $f_{ij} = 0$ for $i \leq k, j \geq k$. But this means that the rank of the matrix $f_{ij}$ is not maximal, which contradicts linear independence of the vectors $u_i(0)$. So, the proposition is proved. 
%
\end{proof}



\section{Jordan--Kronecker invariants of Lie algebra representations}
In this section we define the main object of the present paper: Jordan--Kronecker invariants of Lie algebra representations.\par
Consider a finite-dimensional linear representation  $\rho: \goth g \to \gl(V)$ of a finite-dimensional Lie algebra $\goth g$.  To each point  $x\in V$, the representation $\rho$ assigns a linear operator  $R_x: \goth g\to V$, $R_x(\xi) =  \rho(\xi) x \in V$.  Since the mapping $x \mapsto R_x$ is in essence equivalent to $\rho$,  many natural algebraic objects related to $\rho$ can be defined in terms of $R_x$.

\begin{example}
The \textit{stabilizer} of $x \in V$  can be defined as
\[
\St_x = \Ker R_x = \{\xi\in\goth g~|~ R_x(\xi)=\rho(\xi) x =0 \}\subset \goth g.
\]
\end{example}

Now consider the pencil of such operators generated by a pair of vectors  $a, b\in V$. By the {\it algebraic type} of a pencil $R_a + \lambda R_b=R_{a+\lambda b}$, we will understand the following collection of discrete invariants:

\begin{itemize}

\item the number of distinct eigenvalues of Jordan blocks,

\item the number and sizes of the Jordan blocks associated with each eigenvalue,

\item horizontal and vertical  indices.

\end{itemize}

\begin{proposition}
\label{A:GenLinCombGen}
The algebraic type of a pencil $R_a + \lambda R_b$ does not change under replacing $a$ and  $b$ with any linearly independent combinations of them $a'=\alpha a+\beta b$  and $b'=\gamma a + \delta b$.
\end{proposition}

In other words, the type characterizes two-dimensional subspaces in  $V$ or,  which is the same, one-dimensional projective subspaces in the projectivization of $V$.

\begin{proof}
It is easy to see that replacing two operators with their independent linear combinations does not change the discrete invariants in the Jordan-Kronecker normal form (only eigenvalues of the Jordan blocks do change).
\end{proof}

Since the number of different algebraic types is finite, it is easily seen that in the space $V\times V$ there exists a non-empty Zariski open subset of pairs $(a,b)$ for which the algebraic type of the pencil $R_{a + \lambda b}$ will be one and the same (cf. \cite[Proposition 1]{BolsZhang}).

\begin{definition}
\label{generic}
A pair $(x, a)\in V\times V$ from this subspace and the corresponding pencil $R_{a+\lambda b}$ will be called {\it generic}.
\end{definition}

\begin{definition}
{\it The Jordan--Kronecker invariant} of $\rho$ is the algebraic type of a~generic pencil $R_{a + \lambda b}$.
\end{definition}

In particular, horizontal and vertical indices of a generic pencil will be denoted by
 $\mathsf h_1(\rho),\dots,\mathsf h_p(\rho)$ and $\mathsf v_1(\rho), \dots, \mathsf v_q(\rho)$ and will be called {\it horizontal and vertical indices} of the representation $\rho$.

\section{Interpretation of Jordan-Kronecker invariants} \label{S:InterpretJK}

In this section we give an interpretation for some of the Jordan-Kronecker invariants. All properties we discuss here are quite elementary but will be useful in the sequel. We begin with the Jordan part.
\begin{definition}
A point $a \in V$ is called {\it regular}, if
\[
\dim \St_a \leq \dim \St_x \quad \text{for all } x \in V.
\]
Those points which are not regular are called {\it singular}.\end{definition}

 The set of singular points will be denoted by $\Sing \subset V$. In terms of  $R_x$ we have
$$
\Sing = \{ y\in V~|~ \rk R_y < r= \max_{x\in V} \rk R_x\}.
$$
The dimension of the stabilizer of a regular point is a natural characteristic of $\rho$ and we will denote it by $\dimSt$.  Though in our paper we never use the action of the Lie group  $G$ associated with the Lie algebra $\goth g$, it will be convenient to keep in mind the action and its orbits.  We will need, however, not the orbits themselves but their dimensions only. In particular,  for the dimension of a regular orbit we will use the notation $\dimO$. Notice that
$$
T_x\mathcal O_x = \mathrm{Im}\, R_x \quad \mbox{and} \quad \dim \mathcal O_x = \rk R_x.
$$



\begin{proposition} \label{Prop:EigenSing}
\begin{enumerate}
\item The eigenvalues of Jordan blocks of a pencil $R_{a+\lambda b}$ are those values of  $\lambda\in\mathbb C$ for which the line $a+ \lambda b$ intersects the singular set $\Sing$. In particular, $R_{a+\lambda b}$ has a Jordan block with eigenvalue $\infty$ if and only if $a$ is singular.
\item A generic pencil $R_a + \lambda R_b$ has no Jordan blocks if and only if the codimension of the singular set  $\Sing$ is greater or equal than  $2$.
\end{enumerate}
\end{proposition}

\begin{proof}
The first statement is an immediate corollary of Proposition \ref{cor2}. Furthermore, it follows from Proposition \ref{cor2} that a generic pencil $R_a + \lambda R_b$ has no Jordan blocks if and only if all these operators are of the same rank, i.e. a generic line $a+\lambda b$  does not intersect the singular set $\Sing$.
Clearly, the latter condition is fulfilled if and only if $\codim \Sing \geq 2$.
\end{proof}


 Let us discuss the case $\codim\Sing =1$  in more detail. Consider the matrix of the operator $R_x$ and take all of its minors of size $r\times r$, where $r=\dimO$, that do not vanish identically (such minors certainly exist). We consider them as polynomials  $p_1(x),\dots, p_N(x)$ on $V$. The singular set  $\mathsf{Sing}\subset V$ is then given by the system of polynomial equations
$$
p_i(x)=0,  \quad i=1,\dots, N.
$$

This set is of codimension one if and only if these polynomials possess a non-trivial greatest common divisor which we denote by $\mathsf{p}_\rho$.

Thus, we have $p_i(x) = \mathsf{p}_\rho (x) h_i(x)$,   which implies that the singular set  $\Sing$ can be represented as the union of two subsets:
\begin{equation}\label{sing01}
\mathsf{Sing}_0 =\{ \mathsf{p}_\rho (x) =0\} \quad \mbox{and} \quad \mathsf{Sing}_1 =\{ h_i (x) =0, \  i=1,\dots, N\}.
\end{equation}

It is easy to see that  $\mathsf{p}_\rho (x)$ is a semi-invariant of the representation $\rho$.  This follows from the fact that the action of $G$ leaves the singular set $\mathsf{Sing}_0$ invariant and therefore may only multiply $\mathsf{p}_\rho$ by a character of $G$.  We will refer to this polynomial $\mathsf{p}_\rho$ as the {\it fundamental semi-invariant} of  $\rho$. 
The fundamental semi-invariant is closely related to the characteristic polynomial  $\charp_{a,b}$ of the pencil $R_{a+\lambda b}$:

\begin{proposition}\label{prop:fundinv}
Let $a,b\in V$ be such that the projective line $\alpha a+ \beta b$ does not intersect $\Sing_1$ and is not completely contained in $\Sing$ (i.e. contains at least one regular element). Then
$$
\charp_{a,b}(\alpha, \beta) = \mathsf{p}_\rho(\alpha a+ \beta b).
$$
\end{proposition}
\begin{proof}
As above, let $p_1(x), \dots, p_N(x)$ be the $r\times r$ minors of the matrix $R_x$. Then $p_i(x) = \mathsf{p}_\rho (x)h_i(x)$ for certain polynomials $h_i(x)$. Substituting $x = \alpha a+ \beta b$, we get
$$p_i(\alpha a + \beta b) = \mathsf{p}_\rho (\alpha a + \beta b)h_i(\alpha a + \beta b).$$ Further, notice that since the line $\alpha a+ \beta b$ contains a regular element, if follows that the rank of the pencil $R_{a + \lambda b}$ is equal to $\dimO$. Therefore, $\charp_{a,b}$ is, by definition, the greatest common divisor of the expressions $p_i(\alpha a + \beta b) $, where the latter are viewed as polynomials in $\alpha$ and $\beta$. So,
$$
\charp_{a,b} = \mathrm{gcd}\{ \mathsf{p}_\rho (\alpha a + \beta b)h_i(\alpha a + \beta b)\} = \mathsf{p}_\rho (\alpha a + \beta b)\cdot\mathrm{gcd}\{ h_i(\alpha a + \beta b)\}.
$$
For the sake of contradiction, assume the latter factor is non-constant. Then there exist $\alpha$ and $\beta$, not simultaneously equal to zero, such that $h_i(\alpha a + \beta b) = 0$ for every $i$ (here we use that the polynomials $h_i(\alpha a + \beta b)$ are homogeneous). But this means that the line $\alpha a + \beta b$ intersects $\Sing_1$, which is not the case. So,
$\charp_{a,b} =  \mathsf{p}_\rho (\alpha a + \beta b)$, up to a constant factor.
\end{proof}

\begin{corollary}
\label{aboutsjord}
The degree of the fundamental semi-invariant $\mathsf{p}_\rho$ is equal to the sum of the sizes of all Jordan blocks for a generic pencil $R_{a+\lambda b}$. 
\end{corollary}
\begin{proof}
Indeed, by Proposition \ref{prop:fundinv} for generic $x,a$ we get $\deg \mathsf{p}_\rho = \deg \charp_{a,b}$. But the degree of $\charp_{a,b}$ is exactly the sum of the sizes of Jordan blocks.
\end{proof}

 We now turn to the Kronecker part.
Following Section~\ref{S:JK_Operator},  for an arbitrary pencil $R_{a+\lambda b}$ we define the numbers $\svert(a,b)$ and $\shor(a,b)$.
These numbers computed for a generic pair $(a,b)$ are invariants of the representation $\rho$.  We denote them $\svert(\rho)$,  $\shor(\rho)$ and call
the {\it total Kronecker $v$-index and  $h$-index} of $\rho$. Similarly, let $\nh(\rho), \nv(\rho)$ be the numbers of horizontal and vertical indices for a generic pencil  $R_{a+\lambda b}$.

\begin{proposition}\label{nhnv}
\begin{enumerate} \item The number  $\nh(\rho)$ of horizontal indices of  $\rho$  is equal to $\dimSt$.

\item The number  $\nv(\rho)$ of vertical indices of  $\rho$ is equal to  $\codimO$.
\end{enumerate}
\end{proposition}

\begin{proof}
This follows from $\rk R_{a+\lambda b} = \dim \mathcal O_{a+ \lambda b}$ and Proposition~\ref{cor1}.  \end{proof}

\begin{remark}
\label{aboutjk}
Notice that  (cf. \eqref{sumofsum})
\begin{equation}
\label{sumofsum2}
\svert(\rho) + \shor(\rho) = \dim V  +   \dimSt - \deg \mathsf{p}_{\rho} = \dim\goth g + \codimO -  \deg \mathsf{p}_{\rho}.
\end{equation}
Similarly, if  $\rho$ has an open orbit, i.e. $\codimO =0$, then vertical Kronecker blocks are absent and we have
$\shor(\rho)= \dim \goth g - \deg \mathsf{p}_\rho$. If in addition $\codim\Sing\ge 2$, then $\shor(\rho)= \dim \goth g$.
\end{remark}

 \section{Degrees of invariant polynomials and vertical indices} \label{S:DegInvPolin}

This section contains our main results.
The first result gives a bound for degrees of invariant polynomials in terms of vertical indices.  In the case of the coadjoint representation it was obtained by A.~Vorontsov \cite{Vorontsov}.

\begin{theorem}[Lower bounds for degrees of polynomial invariants]\label{thm1}
\label{T:SumDeg_Polyn} Let $\rho : \g \to \gl(V)$ be a representation of a finite-dimensional Lie algebra $\g$ on a finite-dimensional vector space $V$.
Assume that $f_1, \dots, f_m$ are algebraically independent invariant polynomials of $\rho$, 
and $\deg f_1 \leq \dots \leq \deg f_m$. Let  also $\mathsf v_1(\rho) \leq \dots \leq
\mathsf v_q(\rho)$ be the vertical indices of  $\rho$.
Then 
\begin{equation}
\label{Vorontsovestim}
\deg f_i \geq \mathsf v_i(\rho)
\end{equation}
for $i =1, \dots, m$.
\end{theorem}

\begin{proof}
Let $(a,b) \in V \times V$ be generic.
 Expanding $f_i (a+\lambda x)$ in powers of $\lambda$, we get 
$$
f_i (a+\lambda x) = f_{i0}(x) + \lambda f_{i1}(x) +  \dots + \lambda^{m_i} f_{im_i}(x),  \quad m_i=\deg f_i.
$$
Furthermore, since $f_i$ is an invariant, we have

\begin{equation*}(R_a + \lambda R_b)^* \sum_{j=0}^{m_i} \lambda^j df_{ij}(b) =0,\end{equation*}
and since
$f_{i0}(x) = f_i(a)$, the first term in the latter sum is zero, so we can divide the latter equation by $\lambda$:
$$(R_a + \lambda R_b)^* \sum_{j=0}^{m_i - 1} \lambda^j df_{i,j+1}(b) =0.$$
Also notice that $df_{i,1}(x) =df_i(a)$ for any $x \in V$, so if we take $a$ such that the differentials of $f_i$'s at $a$ are independent, then Proposition~\ref{L:ChainsRestr} gives exactly the desired estimate \eqref{Vorontsovestim}.
\end{proof}
\begin{corollary} Let  $f_1, f_2, \dots, f_q$,  $q=\codimO$, be algebraically independent invariant polynomials of $\rho$. Then
\begin{equation}
\label{sumdeg}
\sum_{i=1}^q \deg f_i  \ge \svert(\rho).
\end{equation}
\end{corollary}
\begin{proof}
This is obtained by adding up inequalities \eqref{Vorontsovestim}.
\end{proof}

Taking into account Remark~\ref{aboutjk}, we also get

\begin{corollary}
\label{indepofdf}
 Let $f_1, f_2, \dots, f_q$,  $q=\codimO$, be algebraically independent invariant polynomials of $\rho$. Suppose that the stabilizer of a regular point is trivial, i.e. $\mathrm{St}_{\mathrm{reg}} = \{0\}$. Then
\begin{equation}
\label{sumdeg2}
\sum_{i=1}^q  \deg f_i \geq \dim V - \deg \mathsf{p}_\rho.
\end{equation}
Moreover, if in addition $\codim\Sing \ge 2$, then
\begin{equation}
\label{sumdeg3}
\sum_{i=1}^q  \deg f_i \geq \dim V.
\end{equation}
\end{corollary}

\begin{remark}
An estimate similar to \eqref{sumdeg3} was obtained by F.~Knop and P.~Littelmann~\cite{Knop87}.
\end{remark}

Another immediate corollary of Theorem \ref{T:SumDeg_Polyn} is the following:

\begin{corollary}\label{sumvsind}
Suppose that there exist algebraically independent invariant polynomials  $f_1, f_2, \dots, f_q$,  $q=\codimO$, of a representation $\rho$ satisfying the condition
\begin{equation}\label{sumCond}
\sum_{i=1}^q \deg f_i  =  \svert (\rho). 
\end{equation}
Then
 \begin{equation}\label{sumCond2}
\mathsf v_i (\rho) =  \deg f_i.
\end{equation}
\end{corollary}
\begin{proof}
Indeed, \eqref{sumCond} is equivalent to  $$\sum_{i=1}^q \deg f_i  = \sum_{i=1}^q \mathsf v_i(\rho),$$
which, in view of \eqref{Vorontsovestim}, is equivalent to \eqref{sumCond2}.
\end{proof}

We now investigate in more detail the case when one of the equivalent conditions~\eqref{sumCond},~\eqref{sumCond2} hold.


\begin{theorem}[On the set where the invariants become dependent]\label{thm2}
Let $\rho : \g \to \gl(V)$ be a representation of a finite-dimensional Lie algebra $\g$ on a finite-dimensional vector space $V$.
Assume that $f_1, f_2, \dots, f_q$  are algebraically independent invariant polynomials of $\rho$  and 
$q = \codimO$. Let also  $\mathsf v_1(\rho), \dots, \mathsf v_q(\rho)$ be the vertical indices of  $\rho$.
 Then the following conditions are equivalent:
 
 \begin{enumerate}
 \item The degrees of $f_i$'s are equal to the vertical indices of $\rho$: $\deg f_i = \mathsf v_i(\rho)$.
 \item The sum of the degrees of $f_i$'s is equal to the total vertical index of $\rho$: $\sum \deg f_i = \sum \mathsf v_i(\rho)$.
 \item The set where the differentials $df_1, \dots, df_q$ are linearly dependent has codimension $\geq 2$ in $V$.
 \item The set where the differentials $df_1, \dots, df_q$ are linearly dependent is contained in the set $\Sing_1$, i.e. in the codimension  $\geq 2$ stratum of the set of singular points of $\rho$ in $V$.
 \end{enumerate}

\end{theorem}

The proof is based on the following lemma, which, in particular, gives an interpretation of the total Kronecker indices:

\begin{lemma}\label{factorizationLemma}
Let $r = \dimO$. For $x \in V$, consider the operator $\Lambda^r R_x : \Lambda^r\g \to \Lambda^r V$,  where $\Lambda^r R_x(\xi_1\wedge\dots\wedge \xi_r)=R_x(\xi_1)\wedge\dots\wedge R_x(\xi_r)$. Then
\begin{equation}\label{wedgeProductDecomp}
\Lambda^r R_x = \mathsf{p}_\rho(x) \cdot \omega_h(x) \otimes \omega_v(x),
\end{equation}
where $\omega_h, \omega_v$ are homogeneous polynomials in $x$ with values in $ \Lambda^r\g^*$,  $\Lambda^r V$ respectively. The polynomials $\omega_h, \omega_v$ are defined uniquely up to constant factors, not divisible by any non-trivial scalar polynomial of $x$, and are of degrees
\begin{align}\label{omegaDeg}
\begin{aligned}
\deg \omega_h &= \shor(\rho) - \nh(\rho) =  \shor(\rho) - \dimSt,\\
\deg \omega_v &= \svert(\rho) - \nv(\rho) =  \svert(\rho) - \codimO.
\end{aligned}
\end{align}
\end{lemma}
\begin{remark}
Formulas \eqref{omegaDeg} can be also rewritten as
$$
\deg \omega_h = \sum_{i=1}^{\nh(\rho)} \bigl(\mathsf h_i(\rho) - 1\bigr), \quad \deg \omega_v = \sum_{i=1}^{\nv(\rho)} \bigl(\mathsf v_i(\rho) - 1\bigr),
$$
i.e. the degrees of $\omega_h, \omega_v$ are given by sums of minimal indices. Also note that two different formulas for each of the degrees in \eqref{omegaDeg} are equivalent due to Proposition \ref{nhnv}.
\end{remark}
\begin{proof}[Proof of Lemma \ref{factorizationLemma}]
The matrix entries of the operator $\Lambda^r R_x$ are $r \times r$ minors of the matrix $R_x$. The greatest common divisor of those minors is, by definition, the fundamental semi-invariant  $\mathsf{p}_\rho(x)$. So, $
 \Lambda^r R_x = \mathsf{p}_\rho(x)  S(x)
 $, where $S(x)$ is a polynomial in $x$ with values in $\Hom( \Lambda^r\g, \Lambda^r V)$. Further, note that the rank of $R_x$ for regular $x$ is $r = \dimO$, so $\dim\Im R_x =r $, and hence the space $\Im \Lambda^r R_x =  \Lambda^r \Im R_x$ is generically one-dimensional. Therefore, the image of $S(x)$ is generically one-dimensional too. So, if we regard $S(x)$ as a matrix over the field of rational functions in $x$, its image has dimension $1$, which means that it is decomposable:
 $$
 S(x) =  \omega_h(x) \otimes \omega_v(x),
 $$
 where $\omega_h$, $\omega_v$ are rational functions in $x$ with values in $ \Lambda^r\g^*$ and $\Lambda^r V$ respectively. Furthermore, multiplying, if necessary, $\omega_h$ by a scalar rational function of $x$ and dividing  $\omega_v$ by the same function, we can arrange that $\omega_h$ is polynomial, and that the greatest common divisor of its coefficients is equal to $1$. But then  $\omega_v$ must be polynomial too, because the product $ \omega_h \otimes \omega_v = S$ is polynomial. Hence, existence of factorization \eqref{wedgeProductDecomp} is proved.\par
  To prove that $\omega_h$, $\omega_v$ are not divisible by any non-trivial scalar polynomial, recall that $\mathsf{p}_\rho$ is the greatest common divisor of the matrix entries of $ \Lambda^r R_x$. Therefore, the greatest common divisor of the matrix entries of $S = \frac{1}{\mathsf{p}_\rho}{ \Lambda^r R_x} $ is $1$, and hence the same is true for its factors $\omega_h$, $\omega_v$, as desired.\par
%
To prove uniqueness, we use that the representation of a rank $1$ operator as a tensor product is unique up to multiplying the first factor by an element of the base field, and dividing the second factor by the same element. Therefore, if we have another representation
%
$$
S(x) = \omega'_h(x) \otimes \omega'_v(x),
 $$
 then there exists a scalar rational function $\mu(x)$, such that 
 $$ \omega'_h(x) =\mu(x)\omega_h(x), \quad  \omega'_v(x) = \frac{1}{\mu(x)}\omega_v(x).$$
 If, moreover, $\omega'_h, \omega'_v$ are polynomials, then it follows that $\omega_h$ is divisible by the denominator of $\mu$, while  $\omega_v$ is divisible by the numerator of $\mu$ (we assume that $\mu$ is written in the reduced form). But we know that $\omega_h$, $\omega_v$ do not have non-trivial polynomial factors. Therefore, $\mu(x)$ must be constant, which proves that factorization~\eqref{wedgeProductDecomp} is unique. This also shows that the polynomials $\omega_h, \omega_v$ are homogeneous, because if they were not, then the expression $ \omega_h(\lambda x) \otimes \omega_v(\lambda x)$, divided by a suitable constant, would provide another factorization of $S(x)$. \par
 Now it remains to compute the degrees of $\omega_h ,\omega_v$. To that end, we restrict~\eqref{wedgeProductDecomp} to a $2$-plane  $x = \alpha a + \beta b$, where $a, b \in V \times V$ are generic. This gives
 \begin{equation}\label{wedgeProductDecomp2}
\Lambda^r R_{ \alpha a + \beta b} =\charp_{a,b}(\alpha, \beta) \cdot \omega_h(\alpha a + \beta b) \otimes \omega_v(\alpha a + \beta b),
\end{equation}
where we used Proposition \ref{prop:fundinv} to rewrite the first factor in the right-hand side. 
Now, consider the Jordan-Kronecker normal form of the pencil $R_{ \alpha a + \beta b}$. Take all columns of all vertical Kronecker blocks of non-zero width (these can be naturally viewed as vectors in $V$). In addition to that, take basis vectors in $V$ corresponding to rows of all Jordan blocks and all horizontal Kronecker blocks of non-zero height. Taken together, all these vectors form a basis in $\Im R_{ \alpha a + \beta b}$ for generic $\alpha, \beta$. Thus, their wedge product, which we denote by $f(\alpha, \beta)$, is a non-trivial polynomial of $\alpha, \beta$ valued in $\Im \Lambda^r R_{ \alpha a + \beta b}$. At the same time, by \eqref{wedgeProductDecomp2} the latter space  is generated by $\omega_v(\alpha a + \beta b)$, so we must have
$$
f(\alpha, \beta) = \mu(\alpha, \beta)\omega_v(\alpha a + \beta b),
$$
where $\mu$ is a rational function. Note also that since $\omega_v(x)$ is not divisible by a non-trivial scalar polynomial, it follows that its vanishing set in $\mathbb{P}V$ has codimension at least $2$. A generic projective line $\alpha a + \beta b$ does not intersect this set, which means that $\omega_v(\alpha a + \beta b)$ does not vanish at all (unless $\alpha = \beta = 0$), and hence $\mu(\alpha, \beta)$ is actually a polynomial function. This gives
$$
\deg \omega_v(\alpha a + \beta b) \leq \deg f(\alpha, \beta) = \svert(\rho) - \nv(\rho),
$$
and thus
\begin{equation}\label{omegaDegIneq1}
\deg \omega_v(x) \leq  \svert(\rho) - \nv(\rho),
\end{equation}
where we used that $\deg f = \svert(\rho) - \nv(\rho)$ by construction. 
Furthermore, an analogous argument for the dual pencil gives
\begin{equation}\label{omegaDegIneq2}
 \deg \omega_h(x) \leq  \shor(\rho) - \nh(\rho).
\end{equation}
At the same time
$$
\deg \omega_v +  \deg \omega_h = \deg \Lambda^r R_{x} - \deg  \mathsf{p}_\rho = \dimO - \deg  \mathsf{p}_\rho,
$$
while the sum of right-hand sides in \eqref{omegaDegIneq1} and \eqref{omegaDegIneq2} is
\begin{align*}
\svert(\rho) +   \shor(\rho) -  \nv(\rho) - \nh(\rho) = \svert(\rho) +  \shor(\rho) - \dimSt - \codimO \\ = \svert(\rho) +  \shor(\rho) - \dimSt + \dimO - \dim V =  \dimO - \deg  \mathsf{p}_\rho,
\end{align*}
where the first equality uses Proposition \ref{nhnv}, while the last one uses \eqref{sumofsum2}. So, the sum of left-hand sides in \eqref{omegaDegIneq1} and \eqref{omegaDegIneq2} is equal to the sum of right-hand sides, and thus both inequalities must be equalities, providing the desired formulas for $\deg \omega_v$, $\deg \omega_h$.
\end{proof}

\begin{proof}[Proof of Theorem \ref{thm2}]
Equivalence of Conditions 1 and 2 follows from Corollary \ref{sumvsind}. Also, implication 4 $\Rightarrow$ 3 is obvious. So, it suffices to prove implications  2 $\Rightarrow$ 4 and 3 $\Rightarrow$ 2. We begin with 2 $\Rightarrow$ 4. Let $\Omega \in \Lambda^{\dim V} V$ be a constant non-zero top-degree multi-vector on $V$ (i.e. a top-degree form on $V^*$). Let also $r = \dimO$. Define an $r$-vector $\omega$ on $V$ by
$$
\omega(x)(\xi_1, \dots, \xi_r) = \Omega(df_1(x), \dots, df_q(x), \xi_1, \dots, \xi_r).
$$
In other words, $\omega$ is a multi-vector dual to the form $df_1 \wedge \dots \wedge df_q$. Note that for generic $x$ the differentials $df_1(x), \dots, df_q(x)$ span the annihilator of the tangent space $T_x\mathcal O_x$ to the orbit of $\rho$. Therefore, $\omega \in \Lambda^r T_x\mathcal O_x$. But the latter space is generically one-dimensional and generated by the form $\omega_v(x)$ given by \eqref{wedgeProductDecomp}. Therefore, we must have
\begin{equation}\label{omegaVSomegav}
\omega(x) = \mu(x) \omega_v(x),
\end{equation}
where $\mu(x)$ is a rational function. But the form $\omega_v$ is not divisible by any scalar polynomial (see Lemma \ref{factorizationLemma}), so $\mu(x)$ is in fact a polynomial. Furthermore, we have
\begin{equation}\label{degOmega}
\deg \omega = \sum \deg df_i = \sum (\deg f_i - 1) =  \left(\sum \deg f_i\right) - \codimO.
\end{equation}
So, assuming Condition 2 of the theorem, i.e. $\sum \deg f_i = \svert(\rho)$, we get
$$
\deg \omega  = \svert(\rho) - \codimO.
$$
But the latter number is equal to $\deg \omega_v$ by Lemma \ref{factorizationLemma}, so we must conclude that $\mu(x)$ in~\eqref{omegaVSomegav} is a constant function, and zeros of $\omega(x)$ are the same as zeros of $\omega_v(x)$. Furthermore, zeros of $\omega$ are exactly those points where $df_1, \dots, df_q$ become linearly dependent, while zeros of $\omega_v$ are contained in the set of zeros of  $\omega_h(x) \otimes \omega_v(x) = \frac{1}{\mathsf{p}_\rho}{ \Lambda^r R_x}$. But the latter set is exactly $\Sing_1$, which proves the implication 2 $\Rightarrow$ 4. \par
We now prove 3 $\Rightarrow$ 2. Condition 3 says that $df_1, \dots, df_q$ are linearly dependent on a set of codimension $\geq 2$, which is equivalent to saying that the zero set of $\omega(x)$ has codimension $\geq 2$. But this is only possible if $\mu(x)$ in \eqref{omegaVSomegav} is a constant function, in which case we have $\deg \omega = \deg \omega_v$. In view of \eqref{degOmega}, this gives
$$
\sum \deg f_i = \deg \omega_v + \codimO = \svert(\rho),
$$
as desired.
\end{proof}

As can be seen from the proof, in general the set $\Sing_1$ consists of two components: the zeros of $\omega_v$ (which, under the conditions of Theorem \ref{thm2}, coincide with the set where the invariants become dependent), and the zeros of $\omega_h$. This immediately gives the following:

\begin{corollary}\label{coadj}
Let $\g$ be a finite-dimensional Lie algebra, and let $\rho : \g \to \gl(V)$ be either its coadjoint representation, or any finite-dimensional representation such that the stabilizer of a generic element in $V$ is trivial. 
Assume that $q=\codimO$ and $f_1, f_2, \dots, f_q$ are algebraically independent invariant polynomials of $\rho$ whose degrees are equal to the vertical indices: $\deg f_i = \mathsf v_i(\rho)$ (equivalently, $\sum \deg f_i = \sum \mathsf v_i(\rho)$). Then their differentials $df_1, df_2, \dots, df_q$ are linearly dependent exactly on the set $\Sing_1$.
\end{corollary}

\begin{remark}
In the case of the coadjoint representation, this is equivalent to a result of D.Panyushev  (Theorem~1.2. in  \cite{PY1}), proved in the case $\codim \Sing \geq 2$. \end{remark}
\begin{proof}[Proof of Corollary \ref{coadj}]
Indeed, for the coadjoint representation due to skew-symmetry of $R_x$ we have $\omega_h = \omega_v$, while in the trivial stabilizer case we have that $ \omega_h $ is of degree $0$ and hence constant. In both cases, the zero set of $\omega_h \otimes \omega_v$ is the same as the zero set of $\omega_v$, hence the result.\end{proof}

\begin{theorem}[On polynomiality of the algebra of invariants]\label{thm3}
Let $\rho : \g \to \gl(V)$ be a representation of a finite-dimensional Lie algebra $\g$ on a finite-dimensional vector space $V$.
Assume that $f_1, f_2, \dots, f_q$  are algebraically independent invariant polynomials of $\rho$ and $q=\codimO$. Let also  $\mathsf v_1(\rho), \dots, \mathsf v_q(\rho)$ be the vertical indices of  $\rho$.
Then we have the following:
 \begin{enumerate}
 \item If the degrees of $f_i$'s are equal to the vertical indices: $\deg f_i = \mathsf v_i(\rho)$ (equivalently, $\sum \deg f_i = \sum \mathsf v_i(\rho)$), then the algebra $\mathbb C[V]^{\goth g}$ of polynomial invariants of $\rho$ is freely generated by $f_1, f_2, \dots, f_q$ (i.e. it is a polynomial algebra).
 \item Conversely, if the algebra $\mathbb C[V]^{\goth g}$ of polynomial invariants of $\rho$ is freely generated by $f_1, f_2, \dots, f_q$, and, in addition, $\rho$ has no proper semi-invariants (i.e. any semi-invariant is an invariant), then the degrees of $f_i$'s are equal to the vertical indices.
 

 \end{enumerate}

\end{theorem}
 \begin{remark}
The condition on semi-invariants holds, for example, when $\g$ is perfect, i.e. coincides with its derived subalgebra. This condition cannot be omitted, as shown by the following example. Consider a $3$-dimensional Lie algebra with relations $[z,x] = x, [z,y] = -2y$. Then the algebra of invariants of its coadjoint representation is freely generated by a degree $3$ function $x^2y$, while the only vertical index is equal to $2$ (and, as predicted by Theorem \ref{thm2}, the gradient of the invariant vanishes on a hypersurface). Theorem \ref{thm3} does not apply in this case because $x$ and $y$ are proper semi-invariants.
 \end{remark}
 
  \begin{remark}
A similar result is obtained in \cite{JS} (see Corollary 5.5):  if the algebra of polynomial invariants is freely generated by  $f_1, f_2, \dots, f_q$, and, in addition, a certain semi-invariant is an invariant, then $\sum \deg f_i$ is equal to the degree of a certain form. A novel feature of our work is that we, first, compute the degree of the latter form (and hence show that the number $\sum \deg f_i$  is equal to the total vertical index), and second, provide formulas for the degrees $\deg f_i$ themselves, and not just for their sum.

 \end{remark}
\begin{proof}[Proof of Theorem \ref{thm3}]
Assume that the degrees of $f_i$'s are equal to the vertical indices. Take any polynomial invariant $f \in \mathbb C[V]^{\goth g}$ of $\rho$. Then, since $\trdeg \mathbb C[V]^{\goth g} = q$, it follows that $f $ is algebraically dependent with $f_1, f_2, \dots, f_q$. Given also that $f_1, f_2, \dots, f_q$ are independent outside of a subset of codimension $\geq 2$ (Theorem \ref{thm2}), it follows from Theorem 1.1 of \cite{PPY} that $f$ is a polynomial function of $f_1, f_2, \dots, f_q$. So, the polynomials $f_1, f_2, \dots, f_q$ generate $\mathbb C[V]^{\goth g}$, and since those polynomials are independent, the algebra is freely generated, as desired.\par
 
 Conversely, assume that $\mathbb C[V]^{\goth g}$ is freely generated by $f_1, f_2, \dots, f_q$. Assume, for the sake of contradiction, that $df_1, df_2, \dots, df_q$ are linearly dependent on a subset $S \subset V$ of codimension~$1$. Then the union of codimension $1$ irreducible components of $S$ is given by a single polynomial equation $h_f(x) = 0$. Since $f_1, f_2, \dots, f_q$ are invariants, it follows that $h_f$ is a semi-invariant. So, under the assumptions of the theorem, $h_f$ must be an invariant, and hence a scalar by Proposition 5.2. of \cite{JS}. Therefore, $f_1, f_2, \dots, f_q$ are actually dependent on a subset of codimension $\geq 2$, and the result follows by Theorem \ref{thm2}. \end{proof}

\section{JK invariants for standard representations of simple Lie algebras}
\label{S:SumsStandRepr}

In this section we consider classical simple Lie algebras, namely $\operatorname{sl}(n)$, $\operatorname{so}(n)$, and $\operatorname{sp}(n)$. For each of them we calculate the JK invariants for the sums of their standard representations. In each case we give the answer first and then show how it correlates with the general results from Sections~\ref{S:InterpretJK} and \ref{S:DegInvPolin}, most importantly, Theorems~\ref{T:SumDeg_Polyn}, \ref{thm2} and \ref{thm3} about vertical indices and the algebra of invariants.

In all cases we consider a matrix subalgebra $\mathfrak{g} \subset \operatorname{gl}(n)$ and the sum of its $m$ standard representations \[\rho^{\oplus_m}: \mathfrak{g} \to \operatorname{gl}(V^{\oplus_m}). \] Fix a basis of $V$, so that we could identify all the linear maps with matrices. Then an element of $V^{\oplus_m}$ is given by a matrix $X \in \operatorname{Mat}_{n \times m}$ and it defines a linear mapping  \begin{equation} \label{Eq:gl_n_StandReprSum_MatrixForm} \begin{gathered}  R_X: \mathfrak{g} \to \operatorname{Mat}_{n \times m}, \\ Y \to YX. \end{gathered} \end{equation}

\begin{lemma} \label{L:GenMatPair} Consider the sum of $m$ standard representations for any of the Lie algebras $\operatorname{sl}(n)$, $\operatorname{so}(n)$, or  $\operatorname{sp}(n)$. Define $n \times m$  matrices $X$ and $A$ as follows:  \begin{itemize}

\item If $m<n$, then \begin{equation} \label{Eq:GenMatPairRowMoreCol}  X = \left( \begin{matrix} 0 \\  I_m \end{matrix} \right), \qquad A = \left( \begin{matrix} -I_m \\ 0 \end{matrix} \right).\end{equation}

\item If $n=m$, then \begin{equation} \label{Eq:GenMatPairRowEqCol}    X = \left( \begin{matrix} \lambda_1 & & \\ & \ddots & \\ & & \lambda_n \end{matrix} \right), \qquad A = \left( \begin{matrix} -1 & & \\ & \ddots & \\ & & -1 \end{matrix} \right), \end{equation} where $\lambda_i$'s are distinct numbers.

\item If $m>n$, then \begin{equation} \label{Eq:GenMatPairRowLessCol}    X =\left( \begin{matrix} 0 & I_n \end{matrix} \right), \qquad A =\left( \begin{matrix} -I_n & 0 \end{matrix} \right).\end{equation}

Then the pencil $R_{X +\lambda A}$ is generic.

\end{itemize}

\end{lemma}

\begin{proof}  We need to show that there exists an open dense subset $U \subset \operatorname{Mat}_{n \times m} \times \operatorname{Mat}_{n \times m}$ such that for any pair $(\tilde X, \tilde A) \in U$, the pencil $R_{\tilde X +\lambda \tilde A}$ has the same algebraic type as the pencil $R_{X +\lambda A}$. What is well-known is that there exists an open dense subset $U \subset \operatorname{Mat}_{n \times m} \times \operatorname{Mat}_{n \times m}$ such that, for any pair $(\tilde X, \tilde A) \in U$, its orbit under the natural left-right action of $\operatorname{GL}(n)\times \operatorname{GL}(m)$ on $\operatorname{Mat}_{n \times m}$ contains a  pair $(X,A)$ of the above form, see e.g. \cite[p. 119]{Pok}. In other words, any generic pair of operators $\mathbb R^m \to \mathbb R^n$ can be written in the above form $(X,A)$ is a suitable basis. So, to show that the pencil $R_{X +\lambda A}$ is generic, it suffices to prove that for any $C \in \operatorname{GL}(n), D \in \operatorname{GL}(m)$ the JK decompositions of the pencils $R_X +\lambda R_A$ and $R_{CXD} + \lambda R_{CAD}$ coincide. The latter is equivalent to the existence of invertible linear mappings $\varphi : \g \to \g$ and $\psi : \operatorname{Mat}_{n\times m} \to \operatorname{Mat}_{n\times m} $ (where $\varphi$ does not have to be a Lie algebra automorphism) such that the following diagram commutes 


%
%
\[
\begin{CD}
\mathfrak{g}  @>R_{CXD} +  \lambda R_{CAD}>> \operatorname{Mat}_{n \times m} \\
@VV \varphi V @AA\psi A\\
\mathfrak{g} @>R_X + \lambda R_A >> \operatorname{Mat}_{n \times m}
\end{CD}.
\] 
For  $\operatorname{sl}(n)$, a suitable choice of $\varphi, \psi$ is \[\varphi(Y) = C^{-1}YC, \quad \psi(Z) = C ZD. \]Similarly for $\operatorname{so}(n)$ and $\operatorname{sp}(n)$ we take \[\varphi(Y) = C^{*}YC, \quad \psi(Z) = (C^{*})^{-1} ZD, \] where $C^*=C^\top$ for $\operatorname{so}(n)$ and $C^* = \Omega^{-1}C^\top\Omega$ for $\operatorname{sp}(n)$, with $\Omega$ being the matrix of the symplectic form given by formula \eqref{Eq:SympLieAlg_Elem} below. In all cases we see that the automorphisms  $\varphi, \psi$ intertwine the pencils $R_X +\lambda R_A$ and $R_{CXD} + \lambda R_{CAD}$, so those pencils have the same JK type, as desired.
\end{proof}

\subsection{Standard representations of  $\operatorname{sl}(n)$}
\label{SubS:SLStandRepr}


\begin{proposition} \label{T:JKSumStandardSLn} Let $\rho$ be the sum of $m$ standard representations of  $\operatorname{sl}(n)$.
\begin{enumerate}
\item Assume that $m<n$. Let $q$ be the integral part of the quotient $m\,/\, (n-m)$, and $r$ be the remainder. Then the JK invariants of $\rho$ consist of $n(n-m)-1$ horizontal indices \[ \underbrace{q+1, \dots, q+1}_{n(n-m-r) -(q+1)}, \quad \underbrace{q+2, \dots, q+2}_{nr+q},\]

%
%
%
%
\item Assume that $m=n$. Then the JK invariants of $\rho$ consist of one vertical index $\mathsf v_1 =n$ and $n$ distinct eigenvalues with $n-1$ Jordan $1 \times 1$ blocks corresponding to each eigenvalue.

\item Assume that $m>n$. Let $q$ be the integral part of the quotient $n\,/\,(m-n)$, and $r$ be the remainder.
Then the JK invariants of $\rho$ consist of the following $n(m-n)+1$ vertical indices:

\begin{itemize}
 \item If $r \not = 0$, then the vertical indices are 
  \[ \underbrace{q+1, \dots, q+1}_{n(m-n-r) + (q+2) }, \quad \underbrace{q+2, \dots, q+2}_{nr - (q+1)}.\]
 \item If $r = 0$, then the vertical indices are    \[  \underbrace{q, \dots, q}_{q+1}, \quad \underbrace{     q+1, \dots, q+1}_{n(m-n)-q}.\]
   \end{itemize}

\end{enumerate}

\end{proposition}
\begin{proof}
We need to compute the Jordan-Kronecker normal form for explicitly given operators $R_X$, $R_A$, where $X$, $A$ are defined in Lemma~\ref{L:GenMatPair}. This is a tedious, yet direct, calculation.
\end{proof}

Now, let us demonstrate how Theorem \ref{thm3} works in this case.
\begin{proposition}\label{slninv}
Let $\rho : \operatorname{sl}(n) \to \mathrm{gl}( \operatorname{Mat}_{n \times m}) $ be the sum of $m$ standard representations of  $\operatorname{sl}(n)$, and let $X \in \operatorname{Mat}_{n \times m} $. Then: 
\begin{enumerate}
\item If $m < n$, the algebra of invariants of $\rho$ is trivial.
\item If $m = n$, the algebra of invariants is freely generated by the polynomial $\det X$.
\item If $m = n + 1$, the algebra of invariants is freely generated by $n \times n$ minors of $X$.
\item If $m > n + 1$, the algebra of invariants is not freely generated. 
\end{enumerate}

\end{proposition}
\begin{remark}
Of course, these results are well-known. For $m > n$, the algebra of invariants of $\rho$ can be identified with the homogeneous coordinate ring of the Grassmannian $\mathrm{Gr}(n,m)$. If $m = n +1$, then the Grassmannian $\mathrm{Gr}(n,m)$ is the projective space $\mathbb{P}^n$, whose homogeneous coordinate ring is the ring of polynomials in $n+1$ variables, hence freely generated. For $m > n +1$, the homogeneous coordinate ring of the Grassmannian $\mathrm{Gr}(n,m)$ is generated by Pl\"ucker coordinates, subject to Pl\"ucker relations. 

\end{remark}
\begin{proof}[Proof of Proposition \ref{slninv}]
There is a clearly no invariants in the $m < n$ case, because in this case the action of the group $\operatorname{SL}(n)$ on $m$-tuples of vectors in $\mathbb R^n$ has an open orbit. (One could also say that in this case there is no vertical indices and hence no invariants by Theorem \ref{thm1}.) As for the case $m = n$, in this case we have one vertical index equal to $n$, and a polynomial invariant $\det X$ of degree $n$. So, by Theorem \ref{thm3}, the algebra of invariants is freely generated by $\det X$ (which is obvious in this case, because the transcendence degree of the algebra of invariants is equal to $1$, so it must be freely generated). 

Further, when $m = n + 1$, the vertical indices are $n, \dots, n$ ($n+1$ times), which again coincides with the degrees of the $n \times n$ minors of $X$. So, these minors generate the algebra of invariants by Theorem \ref{thm3} (they are clearly independent, as one can find a matrix $X$ with any prescribed values of these minors). 

Finally, consider the case $m > n + 1$. Observe that in this case there may also be no invariants of degree $k < n$. Indeed, assume that $f$ is such an invariant. Then there exist $k$ indices $1 \leq i_1 < \dots < i_k \leq n$ such that $f$ has non-trivial restriction to the subspace of matrices whose all columns, except for the ones with indices $i_1, \dots, i_k$, vanish. But this restriction must be an invariant of the representation  $\operatorname{sl}(n) \to  \mathrm{gl}(\operatorname{Mat}_{n \times k})$, and hence trivial, which is a contradiction. So, any non-trivial invariant of $\rho$ has degree at least $n$.  On the other hand, it is easy to see from Proposition \ref{T:JKSumStandardSLn} that at least one of the vertical indices is less than $n$. So, by the second part of Theorem~\ref{thm3}, the algebra of invariants is not freely generated, as desired (this statement applies in our case, because $ \operatorname{sl}(n) $ is simple and thus none of its representations admit proper semi-invariants).
\end{proof}

\subsection{Standard representations of  $\operatorname{so}(n)$ and $\operatorname{sp}(n)$ }

Let us now consider the Lie algebras $\operatorname{so}(n)$, which is the Lie algebra of skew-symmetic matrices, and $\operatorname{sp}(n)$. For $\operatorname{sp}(n)$ we always assume that $n$ is even $n=2k$ and we denote by $\operatorname{sp}(n)$ what would usually be denoted by $\operatorname{sp}(2k)$, that is the space of $n \times n$ matrices $X$ given by the equation 
\begin{equation} 
\label{Eq:SympLieAlg_Elem} X^\top \Omega  + \Omega X = 0, \qquad \Omega ={\begin{pmatrix}0&I_{\frac{n}{2}}\\-I_{\frac{n}{2}}&0\\
\end{pmatrix}}.
\end{equation} 
In other words, the elements of $\operatorname{sp}(n)$ have the form $ X = \Omega^{-1} S$, where $S$ is symmetric. Note that for both $\operatorname{so}(n)$ and $\operatorname{sp}(n)$ the index $n$ is the dimension of the underlying space. Since these Lie algebras correspond to skew-symmetric and symmetric matrices of the same size, most formulas for them differ by some choices of signs. For this reason it is convenient to write the formulas in terms of the number $\varepsilon$, which is equal to $+1$ for  $\operatorname{sp}(n)$ and $-1$ for $\operatorname{so}(n)$.

\begin{proposition} \label{T:JKSumStandardSOn} Let $\rho$ be the sum of $m$ standard representations of  $\operatorname{so}(n)$ or $\operatorname{sp}(n)$. 

\begin{enumerate}

\item  Assume that $m<n$. Let $q$ be the integral part of the quotient $m\,/\, (n-m)$, and $r$ be the remainder.  Then  the JK invariants of $\rho$ are 

 \begin{itemize}
 
 \item   $\displaystyle \frac{(n-m)(n-m + \varepsilon)}{2}$ horizontal indices: \[ \underbrace{2q+1, \dots, 2q+1}_{\frac{(n-m-r)(n-m-r + \varepsilon)}{2}}, \quad \underbrace{2q+2, \dots, 2q+2}_{(n-m-r)r}, \quad \underbrace{2q+3, \dots, 2q+3}_{ \frac{r(r + \varepsilon)}{2}},\] 

\item $\displaystyle \frac{m(m - \varepsilon)}{2}$ vertical indices $\mathsf v_i =2$.

\end{itemize}

\item Assume that $m=n$. Then, in the $\operatorname{so}(n)$ case, the JK invariants of $\rho$ consist of $\displaystyle \frac{n(n+1)}{2} $ vertical indices:   \[ \underbrace{1, \dots, 1}_{n}, \quad \underbrace{2, \dots, 2}_{\frac{n(n-1)}{2} },\]  while in the $\operatorname{sp}(n)$ case  the JK invariants are  $\displaystyle \frac{n(n-1)}{2}$  vertical indices $\mathsf v_i =2$ and $n$  Jordan $1 \times 1$ blocks with different eigenvalues.

\item Assume that $m>n$.  Then the JK invariants of $\rho$ are $\displaystyle n(m-n) + \frac{n(n - \varepsilon)}{2}$ vertical indices: \[ \underbrace{1, \dots, 1}_{(m-n - \varepsilon)n}, \qquad \underbrace{2, \dots, 2}_{\frac{n(n + \varepsilon)}{2} }.\]

\end{enumerate}
\end{proposition}

The proof of this result is analogous to that of Proposition \ref{T:JKSumStandardSLn}. Now, we use this to study the algebra of invariants. As in the $ \operatorname{sl}(n)$ case, the structure of the algebra of invariants for sums of standard representations of $ \operatorname{so}(n)$ and $ \operatorname{sp}(n) $ is well-known, and we discuss it here for the sole purpose of demonstrating the power of Theorem \ref{thm3}.
%
%

\begin{proposition} \label{Prop:SOnInvar} Let $\rho : \operatorname{so}(n) \to \mathrm{gl}( \operatorname{Mat}_{n \times m})$ be the sum of $m$ standard representations of  $\operatorname{so}(n)$, and let $X \in \operatorname{Mat}_{n \times m} $. Then: 

\begin{enumerate}

\item For $m<n$, the algebra of invariants  of $\rho$ is freely generated by $\frac{m(m+1)}{2}$ pairwise inner products of columns of $X$.

\item For $m \geq n$, the algebra of invariants  of $\rho$ is not freely generated.
\end{enumerate}

\end{proposition}

\begin{proof}[Proof of Proposition~\ref{Prop:SOnInvar}]
The first statement is a direct consequence of Theorem~\ref{thm3}, since all pairwise inner products of columns of $X$ have degree $2$ and hence coincide with vertical indices. To prove the second statement, observe that the same restriction argument as we used in the proof of Proposition \ref{slninv} shows that there can be no invariants of degree $1$. On the other hand, some of the vertical indices are equal to $1$, so  the algebra of invariants is not freely generated by Theorem~\ref{thm3}.
 \end{proof}

\begin{remark}  

For $m\geq n$, the algebra of invariants of $\rho$ is known to be generated by  pairwise inner products of columns of $X$, supplemented by $n \times n$ minors of $X$. It is easy to see that pairwise inner products alone do not generate the algebra of invariants, as any minor of $X$ can be expressed as the square root of the Gram determinant of its columns, which is a non-polynomial function in terms of the pairwise inner products. 

%

 \end{remark} 

%
%
%

%
%

The next statement is proved similarly to Proposition~\ref{Prop:SOnInvar}.

\begin{proposition} \label{Prop:SPnInvar} Let $\rho : \operatorname{sp}(n) \to \mathrm{gl}( \operatorname{Mat}_{n \times m})$  be the sum of $m$ standard representations of  $\operatorname{sp}(n)$, and let $X \in \operatorname{Mat}_{n \times m} $. Then: 

\begin{enumerate}

\item For $m \leq n + 1$, the algebra of invariants  of $\rho$ is freely generated by $\frac{m(m-1)}{2}$ pairwise symplectic products of columns of $X$.

\item For $m > n + 1$, the algebra of invariants  of $\rho$ is not freely generated.\end{enumerate}

\end{proposition}

\begin{remark}  

In contrast to the $\operatorname{so}(n)$ case, the minors of $X$ in the symplectic case can be expressed in terms of pairwise symplectic products. Namely, any minor is equal to the Pfaffian of the symplectic Gram matrix of its columns. So, the algebra of invariants of $\rho$ is generated just by the symplectic products. The relations between these symplectic products for $m > n + 1$ are given by Pl\"ucker relations between minors of $X$.


%

 \end{remark} 

%
%

\end{document}